\documentclass[preprint,12pt]{elsarticle}

\usepackage{graphicx}
\usepackage{amssymb,amsmath}

\usepackage{lineno}

\usepackage{graphicx,amssymb,tikz}
\usetikzlibrary{cd}
\usepackage[bottom]{footmisc}
\usepackage[all]{xy}
\usetikzlibrary{positioning}
\tikzset{main node/.style={circle,fill=white!20,draw,minimum size=.4cm,inner sep=0pt},
            }
\usepackage{xpatch,amsthm}

\makeatletter
   \xpatchcmd{\@thm}{\fontseries\mddefault\upshape}{}{}{} 
\makeatother
\newtheorem{theorem}{Theorem}[section]
\newtheorem{lemma}[theorem]{Lemma}
\newtheorem{proposition}[theorem]{Proposition}

\theoremstyle{definition}
\newtheorem{definition}[theorem]{Definition}
\newtheorem{example}[theorem]{Example}

\newtheorem{remark}[theorem]{Remark}

\numberwithin{equation}{section}

\newcommand{\K}{\mathcal{K}}
\newcommand{\Ll}{\mathcal{L}}
\newcommand{\Z}{\mathcal{Z}}
\newcommand{\M}{\Z_{\K_n}(D^1,S^0)}

\usepackage{hyperref}
\journal{Topology and Its Application}

\begin{document}

\begin{frontmatter}


\title{Genus of The Hypercube Graph And Real Moment-Angle Complexes}



\author{Shouman Das\corref{cor1}}
\ead{shouman.das@rochester.edu}

\cortext[cor1]{Corresponding author}
\address{Department of Mathematics, University of Rochester, Rochester NY USA}

\begin{abstract}
In this paper we demonstrate a calculation to find the genus of the hypercube graph $Q_n$ using real moment-angle complex $\Z_\K(D^1,S^0)$ where $\K$ is the boundary of an $n$-gon. We also calculate an upper bound for the genus of the quotient graph $Q_n/C_n$, where $C_n$ represents the cyclic group with $n$ elements.

\end{abstract}


\end{frontmatter}

\tableofcontents


\section*{Introduction}
In graph theory, the hypercube graph is defined as the 1-skeleton of the $n$-dimensional cube. The graph theoretical properties of this graph has been studied extensively by Harary et al in \cite {Harary}. It is well known that this graph has genus $1+(n-4)2^{n-3}$. This fact was proved by Ringel in \cite{Ringel}, Beineke and Harary in \cite{MR0175805}. The moment-angle complex or polyhedral product has been studied recently in the works of Buchstaber and Panov \cite{BP}, Denham and Suciu \cite{DS}, Bahri et al. \cite{BBCG}. 
In this paper, we give an embedding of the hypercube graph in the real moment-angle complex and calculate the genus of the hypercube graph. This demonstrates an interesting relationship between the geometry of hypercube graph and real moment-angle complex.

\section{Genus of a Graph}

\begin{definition}
The \textit{hypercube graph} $Q_n$ for $n\geq 1$ is defined with the following vertex and edge sets.
\begin{align*}
V &= \{(a_1,\cdots, a_n)\ |\ a_i = 0 \text{ or } 1\}\\ 
&= \textrm{  the set of all ordered binary } n\textrm{-tuples with entries of 0 and 1 }\\
E &= \{(u,v)\in V\times V\ |\ u \text{ and } v \text{ differ at exactly one place}\}
\end{align*}

\end{definition}
It is straightforward to see that the hypercube graph can also be defined recursively as a cartesian product \cite[p.~22]{HararyBook}.
$$
Q_1 = K_2, \quad Q_n = K_2 \square Q_{n-1}.
$$
Now we will define the genus of a graph. In this paper, a `surface' will mean a closed compact orientable manifold of dimension of 2. A graph embedding in a surface means a continuous one to one mapping from the topological representation of the graph into the surface. More explanation about graph embeddings can be found at \cite{TopGraph}. 

\begin{definition}
The \emph{genus} $\gamma(G)$ of a graph $G$ is the minimal integer $n$ such that the graph can be embedded in the surface of genus $n$. In other words, it is the minimum number of handles which needs to be added to the 2-sphere such that the graph can be embedded in the surface without any edges crossing each other.
\begin{figure}
\centering
\includegraphics[scale=0.6]{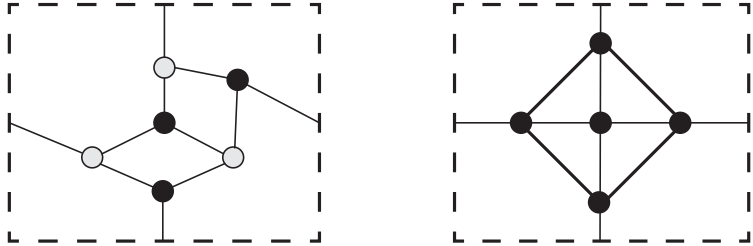}
\label{toroidal}
\caption{$K_{3,3}$ and $K_5$ embedded in a torus.}
\end{figure}

\begin{example}
All planar graphs have genus 0. The complete graph with 5 vertices denoted by $K_5$ and the complete bipartite graph with 6 vertices denoted by $K_{3,3}$ both have genus 1 ( Figure~\ref{toroidal}). The non-planarity of these graphs denoted by $K_n$ and $K_{m,n}$ are explained in \cite[chapter~6]{West}. 
\end{example}

\end{definition}

\begin{definition}[2-cell embedding]
Assume that G is a graph embedded in a surface. Each region of the complement of the graph is called a face. If each face is homeomorphic to an open disk, this embedding is called a 2-cell embedding. 
\end{definition}
In this paper we will restrict our attention to 2-cell embeddings of graphs because the embedding of the hypercube graph in a real moment-angle complex is a 2-cell embedding. We describe it in the next section. 

Now restricting our attention to 2-cell embeddings of a graph $G$ in a surface with genus $g$, we can see that  $$\gamma_M(G) = \max \{g\ |\ G \textrm{ has a 2-cell embedding on a surface
with genus }g\}$$ must exist. This is true because if a graph has a 2-cell embedding in a surface $S_g$, then each handle of the surface must contain at least one edge. So we have a loose upper bound of  $\gamma_M(G)\leq e$ (See \cite{perez} for further explanation). So we can define the maximum genus of a finite connected graph as follows.
\begin{definition}[Maximum genus]
The maximum genus $\gamma_M(G)$ of a connected finite graph $G$ is the maximal integer $m$ such that $G$ has a 2-cell embedding on the surface of genus $m$.
\end{definition}

Two theorems which are important tools in the analysis of graph embeddings follow next.
\begin{theorem}[Euler's Formula] 
Let a graph $G$ has a 2-cell embedding in the surface $S_g$ of genus $g$, with the usual parameter $V, E, F$. Then
\begin{equation}
|V|-|E|+|F|=2-2g
\end{equation}
\end{theorem}
\begin{proof}
See \cite[Chapter~3]{TopGraph}
\end{proof}
\begin{theorem}[Duke\cite{Duke}]
A graph $G$ has a 2-cell embedding in a surface $S_g$ of genus g if and only if $\gamma(G)\leq g \leq \gamma_M(G)$.
\end{theorem}
The last theorem tells us that if there exist 2-cell embeddings of a graph in surfaces of genera $m$ and $n$ with $m\leq n$, then for any integer $k$ with $m\leq k \leq n$, there exists a 2-cell embedding of the graph in a surface with genus $k$. A detailed explanation and proof of this theorem can be found in Richard A. Duke's original paper \cite{Duke}. 

Using these theorems, we can find a lower bound for the genus of the hypercube graph. Let a graph $G$ is embedded in a surface and $f_i$ denote the number of faces which has $i$ edges as its boundary. So we have $$
2|E| = \sum_{i} if_i
$$
For the hypercube graph, each face must have at least 4 edges as its boundary. Therefore, $$
2|E|= \sum_{i\geq4}if_i\geq \sum_{i} 4 f_i=4|F|
$$
which implies $|F|\leq\frac{|E|}{2}$. Now using Euler's formula,
\begin{align*}
g=\ & 1-\frac{|V|}{2}+\frac{|E|}{2}-\frac{|F|}{2}\\
&\geq1-\frac{|V|}{2}+\frac{|E|}{2}-\frac{|E|}{4} = 1-\frac{|V|}{2}+\frac{|E|}{4}
\end{align*}
But for the hypercube graph $Q_n$, we have $|V|=2^n, |E|=n2^{n-1}$. So using the above inequality we get a lower bound\footnote{See \cite{HararyIneq} for more detail on the inequalities involving genus of a graph} for the genus of a hypercube graph,
\begin{equation}
\gamma(Q_n)\geq 1- 2^{n-1}+n2^{n-3}= 1+(n-4)2^{n-3}
\end{equation}

To show that this lower bound can be achieved, we will use the real moment-angle complex. In fact, we prove the following theorem.
\begin{theorem}\label{main}
For $n\geq 3$, the hypercube graph can be embedded in a surface with genus $1+(n-4)2^{n-3}$. Moreover, this embedding is a 2-cell embedding.
\end{theorem}
\section{Moment-Angle Complex}
\begin{definition}
Let $(X,A)$ be a pair of topological spaces and $\K$ be a finite simplicial complex on a set $[m]=\{1,\cdots,m\}$. For each face $\sigma\in\K$, define $$
(X,A)^\sigma = Y_1\times \cdots \times Y_m
$$
where $$
Y_i = \begin{cases}
X & \mathrm{if }\quad i\in \sigma\\
A & \mathrm{if }\quad i\notin \sigma
\end{cases}
$$
The moment-angle complex $\Z_\K(X,A)$ corresponding to pair $(X,A)$ and simplicial complex $\K$ is defines as the following subspace of the cartesian product $X^m$.
$$
\Z_\K(X,A)=\bigcup_{\sigma\in \K} (X,A)^\sigma
$$
\end{definition}
For our calculation we will use the pair $(D^1,S^0)$. This space, $\Z_\K(D^1,S^0)$ is called the real moment-angle complex corresponding to $\K$.
\begin{example}
Let $\Ll_n$ denote the simplicial complex with $n$ discrete points. Then by the above definition
\begin{gather*}
\Z_{\Ll_n}(D^1,S^0) = (D^1\times S^0\times \cdots \times S^0)\cup (S^0\times D^1\times\\
 \cdots \times S^0)\cup \cdots\cup (S^0\times S^0\times \cdots \times D^1) 
\end{gather*}
It is easy to see that $\Z_{\Ll_n}(D^1,S^0)$ is homeomorphic to of the hypercube graph $Q_n$.
\end{example}
From the definition of the moment-angle complex, we can prove the following lemma.
\begin{lemma}
Let $f:\Ll\hookrightarrow \K$ be an inclusion map of simplicial complex where $\Ll$ and $\K$ both have the same number of vertices. Then there exists an inclusion map of moment-angle complexes, $\Z_f: \Z_{\Ll}(X,A)\hookrightarrow \Z_\K(X,A)$.
\begin{proof}
We can consider $\Ll$ as a subcomplex of $\K$. So any face $\sigma$ of $\Ll$ is also a face of $\K$. From this we can conclude that $$(X,A)^\sigma\subset \bigcup_{\tau\in\K}(X,A)^\tau$$
This implies that $\Z_{\Ll}(X,A)\subset\Z_K(X,A)$.
\end{proof}
\end{lemma}
\begin{example}
Let $\K_n$ be the boundary of an $n$-gon and $\Ll_n$ be the $n$ vertices of $\K_n$. Using the above lemma, we can conclude that $\Z_{\Ll_n}(D^1,S^0)=Q_n$ is embedded in $\Z_{K_n}(D^1,S^0)$. Also, if we consider the complement of $\Z_{\Ll_n}(D^1,S^0)$ in $\Z_{K_n}(D^1,S^0)$, we will get a collection of open discs $(D^1\times D^1)^\mathrm{o}$ which is straightforward from the definitions. So, this embedding of $Q_n$ in $\Z_{K_n}(D^1,S^0)$ is clearly a 2-cell embedding.
\end{example}
It is interesting to note that $\Z_{K_n}(D^1,S^0)$ is a closed compact surface with genus $1+(n-4)2^{n-3}$. This fact was proved by Coxeter in \cite{Cox}. We will give an inductive proof here.
\begin{proposition}
For $n\geq 3$, $\Z_{K_n}(D^1,S^0)$ is a closed, compact and orientable surface with genus $1+(n-4)2^{n-3}$.
\begin{proof}
For brevity, we write $\mathcal{Z}_{\K_n}$ to denote $\Z_{\K_n}(D^1,S^0)$.\\
If $n = 3$, it is straightforward that $\Z_{\K_n}=\partial(D^1\times D^1\times D^1)\approx S^2$. 
Let's assume the statement is true for an integer $n\geq 3$. So $\mathcal Z_{\K_n}$ is an orientable surface of genus $1+(n-4)2^{n-3}$. Also note that

\begin{align*}
\mathcal Z_{\K_n} = &\ \ \overbrace {D^1\times D^1\times S^0\times \cdots \cdots \cdots\times S^0}^{n\ \mathrm{ factors}} \\
      & \cup S^0\times D^1\times D^1\times S^0\times \cdots \times S^0 \\
      & \ \ \vdots\\
      & \cup S^0\times \cdots \cdots \cdots \times S^0\times   D^1\times D^1 \\
      & \cup D^1\times S^0 \cdots \cdots \cdots \cdots \times   S^0\times D^1 \\
\end{align*}

Let $B$ be the last term in the union that is $ B =  D^1\times S^0 \times \cdots \times   S^0\times D^1 \subset \mathcal{Z}_{\K_n}$. So $B$ is $2^{n-2}$ copies of $D^1\times D^1$ on the surface $\mathcal{Z}_{\K_n}$. Now note that, $$\partial (B) = (S^0\times S^0 \times\cdots \times   S^0\times D^1) \cup (D^1\times S^0 \times\cdots \times   S^0\times S^0) $$ and
$$
\mathcal{Z}_{\K_{n+1}} = ((\mathcal{Z}_{\K_n}-B)\times S^0)\cup (\partial B\times D^1) 
$$

This means that to construct $\mathcal{Z}_{\K_{n+1}}$, we first delete $2^{n-2}$ copies of $D^1\times D^1$ from $\mathcal{Z}_{\K_n}$, then take two copies of $\mathcal{Z}_{\K_n}-B$ and glue $2^{n-2}$ copies of 1-handle along the boundary of $B$. Therefore,
$$
\mathcal{Z}_{\K_{n+1}} = \mathcal{Z}_{\K_n}\#\mathcal{Z}_{\K_n}\#(2^{n-2}-1) S^1\times S^1 
$$

Here one of the $2^{n-2}$ handles is being used to construct the first connected sum $\Z_{\K_n}\#\Z_{\K_n}$ and the remaining $2^{n-2}-1$ copies of 1-handles are connected as $2^{n-2}-1$ copies of torus. So clearly $\mathcal{Z}_{\K_{n+1}}$ is a closed compact orientable surface with genus $$2(1+(n-4)2^{n-3})+2^{n-2}-1= 1+((n+1)-4)2^{(n+1)-3}.$$ 
\end{proof}
\end{proposition}
In the above discussion, we have proved that the hypercube graph $Q_n$ can be embedded in the real moment-angle complex $\Z_{\K_n}(D^1,S^0)$ which is a surface of genus $1+(n-4)2^{n-3}$. Hence Theorem \ref{main} is proved.

\section{Action of \texorpdfstring{$C_n$}{Cn} on \texorpdfstring{$Q_n$}{Qn} and \texorpdfstring{$\Z_{\K_n}(D^1,S^0)$}{Zk}}

Let $C_n$ denote the cyclic group with $n$ elements generated by $\sigma$. Since $\K_n$ can be considered as the boundary of a regular $n$-gon, we can define an action of $C_n$ on $\K_n$ by rotating the $n$-gon by $2\pi/n$ radians about the centre. If $(i,i+1)$ represents an edge, then this action will take this edge to $(i+1,i+2)$ (here the vertices are considered as $i\ (\mod n)$). So We can define an action of $C_n$ on $\Z_{\K_n}(D^1,S^0)$ by
$\sigma (x_1,\cdots, x_n) = (x_{\sigma(1)},\cdots,x_{\sigma(n)})=(x_2, \cdots, x_n,x_1)$ where $(x_1,\cdots, x_n) \in  (D^1,S^0)^\tau$ for a maximal face $\tau\in \K_n$. So $\sigma$ is rotating the coordinates of a point in the moment-angle complex $\Z_{\K_n}(D^1,S^0)$. We can define a similar action of $C_n$ on the hypercube graph $Q_n$ by rotating the coordinates of a point. It is straightforward to note that the following diagram commutes.
$$
\begin{tikzcd}
Q_n \arrow{d} \arrow[r, hook]
& \M \arrow{d} \\
Q_n/C_n \arrow[r,hook]
& \M/C_n
\end{tikzcd}
$$
Therefore the quotient graph $Q_n/C_n$ is embedded in the quotient space $\M/C_n$. We will show that the quotient space \\ $\Z_{\K_n}(D^1,S^0)/C_n$ is also a closed connected orientable surface. Therefore, calculating the genus of the surface $\Z_{\K_n}(D^1,S^0)/C_n$ would suffice to give an upper bound for the genus of the quotient graph $Q_n/C_n$. First we will prove that $\Z_{\K_n}/C_n$ is closed connected, compact and orientable manifold. Then we will calculate the genus of this surface. Indeed the following theorem can be found in Ali's thesis \cite{Ali} theorem 4.2.2.

\begin{theorem}
Let $C_m$ be the subgroup of $C_n$ i.e. $m|n$, then $\Z_{\K_n}(D^1,S^0)/C_m$ is a closed surface.
\end{theorem}

It is not surprising that the quotient surface $\Z_{\K}(D^1,S^0)/C_n$ must be an orientable surface. We can prove it by giving a $\delta$-complex structure on this surface and check that all the triangles on the surface can be given an orientation such that any two neighboring triangles' edges fit nicely.

\begin{lemma}
The surface $\Z_{\K_n}(D^1,S^0)/C_n$ is an orientable surface.
\begin{proof}
First note that the action of $C_n$ permutes the coordinate of a point in a cyclic manner. So we only need to consider the space $D^1\times D^1\times S^0\times \dots \times S^0$, which is actually $2^{n-2}$ copies of $D^1\times D^1$. For each of these squares, we draw a diagonal from the lower left corner to the top right, make a triangulation of the surface and give suitable orientation to each triangle. Then we glue these squares along their boundaries under the identification generated by $C_n$. Let $\epsilon_1\epsilon_2\dots \epsilon_n$ represent the coordinate $(\epsilon_1,\dots, \epsilon_n )$ where $\epsilon_i = 0 \text{ or } 1$. For abbreviation, we write directed edges as $(000, 010)$ which represent the directed edge from $(0,0,0)$ to $(0,1,0)$.

\textbf{Case 1:}($n=3$).
We have two copies of $D^1\times D^1$ (Figure \ref{fig:my_label5}a). Under the action of $C_n$, we have the following identification of edges on the boundary of this two squares.
$$
(000,010) \sim (000,100),\quad (001,011) \sim (010,110),
$$
$$
(100,110) \sim (001,101),\quad (101,111)  \sim  (011,111) 
$$
As shown in the figure \ref{fig:my_label5}, we can give orientation to to each of the triangles such that the orientation of each edge fits together. 
\begin{figure}
    \centering
    \includegraphics[scale = 0.7]{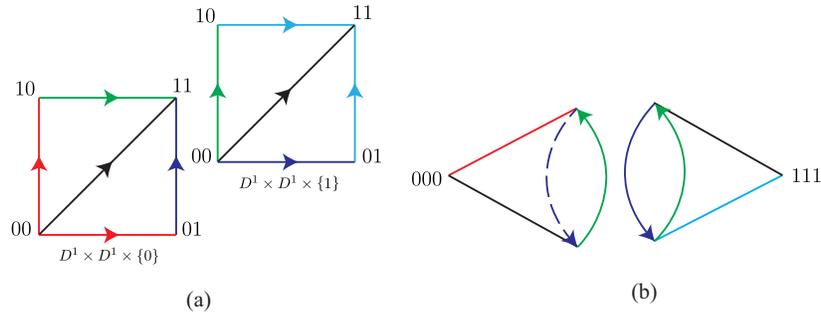}
    \caption{Quotient space $\Z_{\K_3}(D^1,S^0)/C_3$.}
    \label{fig:my_label5}
\end{figure}

\textbf{Case 2:}($n=4$). In this case, we have four copies of $D^1\times D^1$ as shown in (Figure \ref{fig:my_label6}b). And we have the identification of edges as follows:
$$
(0000,0100) \sim (0000,1000),\quad (0010,0110) \sim (0100,1100) ,
$$
$$
(1000,1100) \sim (0001,1001),\quad (1011,1111) \sim (0111,1111) 
$$
$$
(1001,1101) \sim (0011,1011),\quad (0011,0111) \sim (0110,1110),
$$
$$
(1010,1110) \sim (0101,1101),\quad (0001,0101) \sim (0010,1010) 
$$
We can see from the figure that the orientation of each triangle is compatible to each other.
\begin{figure}
    \centering
    \includegraphics[scale = 0.8]{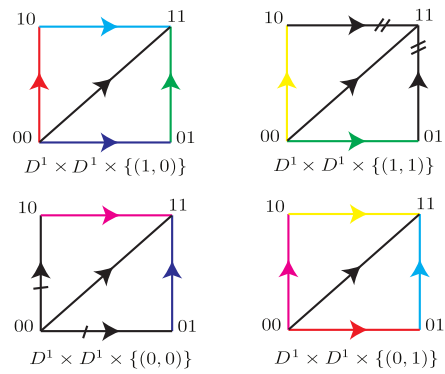}
    \caption{Quotient space $\Z_{\K_4}(D^1,S^0)/C_4$.}
    \label{fig:my_label6}
\end{figure}

\textbf{Case 3:} ($n\geq 5$). For $n\geq 5$, we can have the identification of edges as follows:
$$
 (0000x,1000x) \sim (000x0,010x0),\quad (0010x,0110x) \sim (010x0,110x0),
$$
$$
(100x0,110x0) \sim (00x01,10x01),\quad (101x1,111x1) \sim (01x11,1x111) 
$$
$$
(1001x,1101x) \sim (001x1,101x1),\quad (001x1,011x1) \sim (01x10,11x10),
$$
$$
(1010x,1110x) \sim (010x1,110x1),\quad (000x1,010x1) \sim (00x10,10x10)
$$
Here, $x$ represents a string of length $(n-4)$ whose characters can be $0$ or $1$. So we can see that all this identification preserve the orientation of the surface. 
Therefore, $\Z_\K(D^1,S^0)/C_n$ is an orientable surface.
\end{proof}
\end{lemma}

Since the quotient of a compact and connected space is also a compact and connected space, we have proved the following theorem.
\begin{theorem}
\label{theorem_quotient}
Let $\K$ be the boundary of an $n$-gon. Then $\Z_\K(D^1,S^0)/C_n$ is a closed, compact and orientable surface.
\end{theorem}

\subsection{Branched covering and genus of \texorpdfstring{$\Z_\K(D^1,S^0)/C_n$}{ZK}}

Next, we will prove the following lemma which gives a formula for finding the genus of the quotient space.

\begin{lemma}
\label{my_lemma1}
The genus of $\Z_\K(D^1,S^0)/C_n$ is given by the following formula:
\begin{equation}
g(\Z_\K(D^1,S^0)/C_n) = 1 + 2^{n-3} - \frac{1}{2n}\sum_{d|n} \phi(d)2^{n/d} 
\end{equation}
\end{lemma}
where $\phi$ is the Euler totient function.

To prove this lemma, we will use the Riemann-Hurwitz formula for branched covering. Note that the quotient map $\Z_\K(D^1,S^0) \to \Z_\K(D^1,S^0)/C_n$ would be a covering map if remove a finite number of points (the corners of each $D^1\times D^1$). So this quotient map is a branched cover. We use the following definition from \cite{Turner}.

\begin{definition}
Let $X$ and $Y$ be two surfaces. A map $p: X\to Y$ is called a \textit{branched covering} if there exists a codimension 2 subset $S\subset Y$ such that $p: X\setminus p^{-1}(S) \to Y\setminus S$ is a covering map. The set S is called the branch set and the preimage $p^{-1}(S)$ is called the singular set.
\end{definition}
\begin{definition}
Let $p: X\to Y$ be a branched covering of two surfaces where $Y$ is connected. The degree of this branched covering is the number of sheets of the induced covering after removing the branch points and singular points. 
\end{definition}
In \cite{Ali}, it is proved that $\Z_\K(D^1,S^0) \to \Z_\K(D^1,S^0)/C_n$ is a branched covering of degree $n$. Since the quotient is closed, compact and orientable we can apply the classical Riemann-Hurwitz formula for branched covering. 

\begin{theorem}[Riemann-Hurwitz Formula] Let $G$ be a finite group acting on the surface $X$, such that the map $p:X \to X/G $ be a branched covering with a branch subset $S\subset X/G$. Let $G_y$ represent the isotropy subgroup for a point $y\in X$ and $\chi(X)$ be the Euler characteristic of $X$. Then 
\begin{equation}
\chi (X) = |G|\cdot \chi(X/G) -  \sum_{x\in S} \left(|G|-\frac{|G|}{n_x}\right)
\end{equation}
with $n_x = |G_y|$ for $y\in p^{-1}(x)$ and $x\in S$. 

\end{theorem}
 
 To apply this formula we need to calculate the cardinality of the isotropy subgroup for each of the singular points in $\Z_K(D^1,S^0)$. It is straightforward that the action of $C_n$ on a point permute its coordinate in a cyclic manner. The only points in $\Z_K(D^1,S^0)$ which have a nontrivial isotropy group have coordinates 0's or 1's only. Therefore, the cardinality of the isotropy subgroup is related to the number of aperiodic necklaces with 2-coloring. We will give some necessary definitions related to aperiodic necklaces and then count the Euler characteristic of $\Z_K(D^1,S^0)/C_n$ by using the Riemann-Hurwitz formula.

 \begin{definition}
Let $W$ represent a word of length $n$ over an alphabet of size $k$. We define an action of the cyclic group $C_n=\langle \sigma \rangle$  on W by rotating its characters. For example, if $W = a_1a_2\cdots a_n$ where each $a_i$ is a character from the alphabet, then $\sigma(W) = a_2a_3\cdots a_n a_1$.
 A word $W$ of length $n$ is called an \textit{aperiodic word} if $W$ has $n$ distinct rotation. 
 \end{definition} 

\begin{definition}
An equivalence class of an aperiodic word under rotation is called a \textit{primitive necklace}.
\end{definition} 
The total number of primitive $n$-necklaces on an alphabet of size k, denoted by $M(k,n)$, is given by Moreau's formula \cite{moreau1872permutations},
$$
M(k,n) = \frac{1}{n} \sum_{d|n} \mu (d) k^{n/d}
$$
Note that we can deduce Moreau's formula by using M{\"o}bius inversion formula and the fact that 
$$
k^n = \sum_{d|n} d  M(k,d)
$$

\textbf{Total Number of Necklace of length $n$ with $k$-coloring}: Note that this number is the same as $\sum_{d|n} M(k,d)$ since $M(k,d)$ gives us the number of aperiodic necklaces for each divisor $d$ of $n$. So, we have
$$
\begin{aligned}
\sum_{d|n} M(k,d) &= \sum_{d|n} \frac{1}{d} \sum_{c|d} \mu (c) k^{d/c} \\
&=\frac{1}{n} \sum_{d|n} \sum_{c|d} \frac{n}{d} \mu (d/c) k^{c} \\
&= \frac{1}{n} \sum_{c|n} \sum_{\substack{b|\frac{n}{c}\\ d = bc}} \frac{n}{d} \mu (d/c) k^{c}  \\
&= \frac{1}{n}\sum_{c|n} k^c \sum_{b|\frac{n}{c}}\mu(b)\left(\frac{n}{c}\right) \left(\frac{1}{b}\right)\\
&= \frac{1}{n}\sum_{d|n} \phi(d) k^{n/d}\\
\end{aligned}
$$
The last line follows since $\sum_{b|\frac{n}{c}}\mu(b)(\frac{n}{c}) (\frac{1}{b}) = \phi(n/c)$, where $\phi$ is the Euler's totient function. 

For our calculation, $k=2$ since we are only concerned about words with 2 letters or necklaces with 2-coloring. We will denote this Moreau's formula by 
$$
M(n) = \frac{1}{n} \sum_{d|n} \mu (d) 2^{n/d}
$$
So we have, $\sum_{d|n} M(d)= \frac{1}{n}\sum_{d|n} \phi(d) 2^{n/d}$.
\begin{proof}[\textbf{Proof of lemma \ref{my_lemma1}}]

Note that $C_n$ acts on a point of $\Z_\K(D^1,S^0)$ by cyclically permuting the coordinates. So $C_n$ acts freely on all but finitely many points. The coordinate of those points can be only $0$ or $1$. Each point in the branch set can be considered a primitive necklace of length $d$ where $d|n$. Clearly, there are $M(d)$ points in the branch set which has isotropy group $C_{n/d}$. So the summation in the Riemann-Hurwitz formula becomes $$\sum_{x\in S} \left(|G|-\frac{|G|}{n_x}\right) = \sum_{d|n} M(d) (n- \frac{n}{n/d})$$ 

Now using the Riemann-Hurwitz formula,

$$
\begin{aligned}
&\chi(\Z_\K(D^1,S^0)) = n .\chi(X/G) -  \sum_{d|n} M(d) (n- \frac{n}{n/d})\\
&\implies (4-n)2^{n-2} = n . \chi(X/G) - \sum_{d|n} n M(d) + \sum_{d|n} d M(d)\\
&\implies 2^n -n 2^{n-2} = n. \chi(X/G) - n \sum_{d|n} M(d) + 2^n\\
&\implies  \chi(\Z_\K(D^1,S^0)/C_n) = \sum_{d|n} M(d) - 2^{n-2}\\ 
&\implies \chi(\Z_\K(D^1,S^0)/C_n) = \frac{1}{n}\sum_{d|n} \phi(d)2^{n/d} - 2^{n-2}\\
&\implies g(\Z_\K(D^1,S^0)/C_n) = 1 + 2^{n-3} - \frac{1}{2n}\sum_{d|n} \phi(d)2^{n/d}
\end{aligned}
$$
So the quotient space $Z_{K_n}(D_1,S^0)/C_n$ has genus equal to 
$$ 1 + 2^{n-3} - \frac{1}{2}(\#\textit{ of n-length necklace with 2-coloring}) $$

\end{proof}

\begin{example}
For $n = 6$, $Z_{K_n}(D^1,S^0)$ is a surface with genus $1+ (6-4)2^{6-3} = 17$. Under the above formula, the genus of the quotient space $Z_{K_n}(D^1,S^0)/C_n$ is 
$$
1+ 2^3 - \frac{1}{2}(\textit{\#of 6-length necklace with 2-coloring})
$$
The number of $6$ length necklace with 2-coloring is exactly
$$
\frac{1}{6}\sum_{d|6} \phi(d)2^{6/d} = \frac{1}{6}(1.2^6+1.2^3+2.2^2+2.2) = 14
$$
So $Z_{K_n}(D^1,S^0)/C_n$ has genus $9-14/2 = 2$.

\end{example}

\subsection{ An upper bound for genus of quotient graph \texorpdfstring{$Q_n/C_n$}{Qn/Cn}} From the above discussion, we have proved the following lemma.
\begin{lemma}
The genus of the quotient graph, $Q_n/C_n$ has an upper bound: 
\begin{equation}
\gamma(Q_n/C_n) \leq 1 + 2^{n-3} - \frac{1}{2n}\sum_{d|n} \phi(d)2^{n/d}
\end{equation}
\end{lemma}

\begin{remark}
It can be showed that theorem \ref{theorem_quotient} is also true for a subgroup $C_m\subset C_n$ where $m\lvert n$. In this paper we are not using that result, but I will add a proof of this fact in my PhD thesis. 

\end{remark} 
 
\section*{Acknowledgements}

The author is thankful to Professor Frederick Cohen for his insightful discussion on the real moment-angle complex and related topics. Also many thanks to the reviewer whose careful reading and suggestions have improved the paper.

\bibliographystyle{model1-num-names}
\bibliography{mybib.bib}

\end{document}